\documentclass[11pt,reqno]{amsart}
\usepackage[T1]{fontenc}
\usepackage{amssymb,amscd,amsthm,latexsym}
\usepackage{graphicx}
\usepackage[british]{babel}
\usepackage[all]{xy}
\usepackage{color}
\usepackage{mathabx}
\usepackage{calrsfs}
\usepackage{booktabs}
\usepackage{chngcntr}

\theoremstyle{plain}
\newtheorem{theorem}[subsection]{Theorem}

\newtheorem{proposition}[subsection]{Proposition}
\newtheorem{corollary}[subsection]{Corollary}

\theoremstyle{definition}
\newtheorem{definition}[subsection]{Definition}

\theoremstyle{remark}
\newtheorem{example}[subsection]{Example}
\newtheorem{examples}[subsection]{Examples}
\newtheorem{remark}[subsection]{Remark}

\newdir{ >}{
 @{}*!/-10pt/@{>} }
\newdir{ |>}{
 @{}*!/-5.5pt/@{|}*!/-10pt/:(1,-.2)@^{>}*!/-10pt/:(1,+.2)@_{>} }
\newdir{ >>}{{}*!/8pt/@{|}*!/3.5pt/:(1,-.2)@^{>}*!/3.5pt/:(1,+.2)@_{>}}

\newcommand{\defn}[1]{\textbf{#1}}

\def\mathrmdef#1{\expandafter\def\csname#1\endcsname{{\rm#1}}}
\mathrmdef{max}
\mathrmdef{op}

\newcommand{\C}{\ensuremath{\mathbb{C}}}

\newcommand{\V}{\ensuremath{\mathbb{V}}}
\newcommand{\Grp}{\ensuremath{\mathsf{Grp}}}
\newcommand{\Ord}{\ensuremath{\mathsf{Ord}}}
\newcommand{\Ab}{\ensuremath{\mathsf{Ab}}}
\newcommand{\Mon}{\ensuremath{\mathsf{Mon}}}
\newcommand{\Rng}{\ensuremath{\mathsf{Rng}}}
\newcommand{\SRng}{\ensuremath{\mathsf{SRng}}}
\newcommand{\Set}{\ensuremath{\mathsf{Set}}}
\newcommand{\Pt}{\ensuremath{\mathsf{Pt}}}
\newcommand{\Eq}{\ensuremath{\mathsf{Eq}}}
\newcommand{\W}{\ensuremath{\mathsf{W}}}
\newcommand{\M}{\ensuremath{\mathsf{M}}}
\newcommand{\cod}{\ensuremath{\mathsf{cod}}}
\newcommand{\OrdGrp}{\ensuremath{\mathsf{OrdGrp}}}

\newcommand{\VCat}{\ensuremath{V\mbox{-}\mathsf{Cat}}}

\newcommand{\threesum}[3]{\protect{\left(\begin{smallmatrix} {#1}\\ {#2}\\ {#3}\end{smallmatrix}\right)}}
\newcommand{\twosum}[2]{\protect{\left(\begin{smallmatrix} {#1}\\ {#2}\\ \end{smallmatrix}\right)}}

\def\mathrmdef#1{\expandafter\def\csname#1\endcsname{{\rm#1}}}
\newcommand{\la}{\langle}
\newcommand{\ra}{\rangle}

\newcommand{\ultra}[1]{\dot{#1}}

\def\splitsplitpullback{%
 \ar@{-}[]+R+<6pt,-.5ex>;[]+RD+<6pt,-6pt>%
 \ar@{-}[]+D+<.5ex,-6pt>;[]+RD+<6pt,-6pt>}

\counterwithin*{equation}{section}

\def\mathrmdef#1{\expandafter\def\csname#1\endcsname{{\rm#1}}}
\mathrmdef{max}
\mathrmdef{op}

\newcommand{\App}{\ensuremath{\mathsf{App}}}
\newcommand{\Cat}{\ensuremath{\mathsf{Cat}}}
\newcommand{\Top}{\ensuremath{\mathsf{Top}}}
\newcommand{\Rel}{\ensuremath{\mathsf{Rel}}}
\mathrmdef{sym}

\def\relto{{\longrightarrow\hspace*{-2.8ex}{\mapstochar}\hspace*{2.6ex}}}

\def\xxx{\mathfrak{X}}
\def\xx{\mathfrak{x}}
\def\yy{\mathfrak{y}}
\def\zz{\mathfrak{z}}
\def\ww{\mathfrak{w}}

\begin{document}

\title{A note on Mal'tsev objects}

\author{Maria Manuel Clementino}
\address{CMUC, Department of Mathematics, University of
Coimbra, 3000-143 Coimbra, Portugal}\thanks{}
\email{mmc@mat.uc.pt}

\author{Diana Rodelo}
\address{Department de Mathematics, University of the Algarve, 8005-139 Faro, Portugal and CMUC, Department of Mathematics, University of Coimbra, 3000-143 Coimbra, Portugal}
\thanks{The authors acknowledge partial financial assistance by {\it Centro de Matem\'atica da Universidade de Coimbra} (CMUC), funded by the Portuguese Government through FCT/MCTES, DOI 10.54499/UIDB/00324/2020.}
\email{drodelo@ualg.pt}



\begin{abstract} The aim of this work is to compare the distinct notions of \emph{Mal'tsev object} in the sense of Weighill and in the sense of Montoli-Rodelo-Van der Linden.
\end{abstract}

\maketitle

\section*{Introduction}
A variety of universal algebras $\V$ is called a \emph{Mal'tsev variety}~\cite{Maltsev-Sbornik} when its theory admits a ternary operation $p$ satisfying the equations
\[
	p(x,y,y)=x\;\;\mathrm{and}\;\; p(x,x,y)=y.
\]
Such varieties were characterised in~\cite{Maltsev-Sbornik} by the fact that any pair of congruences $R$ and $S$ on a same algebra $X$ is \emph{2-permutable}, i.e. $RS=SR$. It was shown in~\cite{Lambek} that such varieties are also characterised by the fact that any homomorphic relation $D$ from an algebra $X$ to an algebra $Z$ is \emph{difunctional}: $(x_1 D z_2 \wedge x_2 D z_2 \wedge x_2 D z_1) \Rightarrow x_1 D z_1$.

The notion of Mal'tsev variety was generalised to a categorical context in~\cite{CLP} (see also~\cite{CKP, CPP, BB}). This was achieved by translating the above properties on homomorphic relations for a variety into similar properties on (internal) relations in a category. A regular category~\cite{Barr} $\C$ is called a \emph{Mal'tsev category} when any pair of equivalence relations $R$ and $S$ in $\C$ on a same object $X$ is such that $RS=SR$. Mal'tsev categories can also be characterised by the difunctionality of relations:

\begin{theorem}~\cite{CLP, CPP, CKP}\label{Maltsev char difunctional} A finitely complete category $\C$ is a Mal'tsev category if and only if any relation $D\to X\times Z$ in $\C$ is difunctional.
\end{theorem}

The main examples of Mal'tsev varieties are $\Grp$ of groups, $\Ab$ of abelian groups, $R$-$\mathsf{Mod}$ of modules over a commutative ring $R$, $\Rng$ of rings, and $\mathsf{Heyt}$ of Heyting algebras. More generally, any variety whose theory contains a group operation is a Mal'tsev variety; an example of a non-Mal'tsev variety is the variety $\Mon$ of monoids. As examples of Mal'tsev categories (that are not varietal) we have the category $\Grp(\mathsf{Top})$ of topological groups, any abelian category or the dual of an elementary topos. Also, if $\C$ is a Mal'tsev category, then so are the (co)slice categories $\C/X$ and $X/\C$, for any object $X$ of $\C$ (see~\cite{BB}, for example).

There are also several other well-known characterisations of Mal'tsev varieties through nice properties on relations, such as the fact that every reflexive relation is necessarily a congruence. All of these characteristic properties have been generalised to the Mal'tsev categorical context (see~\cite{CLP, CPP, CKP, BB, BGJ}). This wide range of nice properties together with the long list of examples has contributed to the great amount of research developed on Mal'tsev varieties and categories over the past 70 years. They are also just the first instance of the family of \emph{$n$-permutable} varieties~\cite{HM}, for $n\ge 2$. An $n$-permutable variety is such that any pair of congruences $R$ and $S$ on a same algebra $X$ is $n$-permutable, i.e. $(R,S)_n=(S,R)_n$, where $(R,S)=RSR\cdots$ denotes the composite of $R$ and $S$, $n$ times; they have been generalised to $n$-permutable categories in~\cite{CKP}. Concerning $n$-permutability, the reader may be interested in~\cite{JRVdL}, to see how varietal proofs translate into categorical ones, and~\cite{M-FRVdL} for further properties on relations.

Weighill in~\cite{Weighill} used the characterisation of a Mal'tsev category obtained through the difunctionality of all relations to introduce a definition of Mal'tsev object. To fully understand this definition, we must develop further on relations in $\C$.

A (binary) \emph{relation} $R$ in a category $\C$ from an object $X$ to an object $Z$ is a span $X\stackrel{r_1}{\longleftarrow} R \stackrel{r_2}{\longrightarrow} Z$ such that $(r_1,r_2)$ is jointly monomorphic. We identify two relations $X\leftarrow R \rightarrow Z$ and $X\leftarrow R'\rightarrow Z$, when $R$ factors through $R'$ and vice-versa. When $\C$ admits binary products, a relation $R$ from $X$ to $Z$ is a subobject of $X\times Z$, denoted $\la r_1,r_2\ra\colon R\to X\times Z$. A relation $X\stackrel{r_1}{\longleftarrow} R \stackrel{r_2}{\longrightarrow} Z$ in $\C$ is difunctional when the relation
\begin{equation}\label{W-difunctional}
  \xymatrix@C=40pt{\C(Y,X) & \ar[l]_-{\C(Y,r_1)} \C(Y,R) \ar[r]^-{\C(Y,r_2)} & \C(Y,Z)}
\end{equation}
in $\Set$ is difunctional, for every object $Y$ of $\C$. The definition of a reflexive, symmetric, transitive and equivalence relation in $\C$ is obtained similarly.

\begin{definition}\label{W-Maltsev obj}\cite{Weighill}
An object $Y$ of $\C$ is called a \defn{W-Mal'tsev object}\footnote{We added the prefix ``W-'' to distinguish these from the Mal'tsev objects of Definition~\ref{Maltsev obj}} when for every relation $X\stackrel{r_1}{\longleftarrow} R \stackrel{r_2}{\longrightarrow} Z$ in $\C$, the $\Set$-relation \eqref{W-difunctional} is difunctional.
\end{definition}

It follows from Theorem~\ref{Maltsev char difunctional} that a finitely complete category $\C$ is a Mal'tsev category if and only all of its objects are W-Mal'tsev objects.

It is well-known that the category $(\Set)^\op$ is a Mal'tsev category~\cite{CKP}, hence every set is a W-Mal'tsev object in $(\Set)^\op$. Weighill's study of W-Mal'tsev objects led to the identification of interesting Mal'tsev subcategories of duals of categories of topological flavour. We will recall and generalise these results in Section 2.

Using a completely different approach, Bourn in~\cite{MCFPO} classified several categorical notions (including that of Mal'tsev category) through the fibration of points. A \emph{point} $(f\colon A\to B, s\colon B\to A)$ in a category $\C$ is a split epimorphism $f$ with a chosen splitting $s$. We may define the category of points in $\C$, denoted by $\Pt(\C) $: a morphism between points is a pair $(x,y) \colon(f,s)\to (f',s')$ of morphisms in $\C$ such that the following diagram commutes
\[
\xymatrix@!0@=4em{B \ar[r]^-{s} \ar[d]_y & A \ar[r]^-{f} \ar[d]^-{x} & B \ar[d]^y \\
B' \ar[r]_{s'} & A' \ar[r]_-{f'} & B'.}
\]
When $\C$ has pullbacks of split epimorphisms, the forgetful functor $\cod \colon {\Pt(\C) \to \C}$, which associates
with every split epimorphism its codomain, is a fibration called the \emph{fibration of points}~\cite{Bourn1991}.

From the several classifying properties of the fibration of points $\cod$ studied in~\cite{MCFPO}, we emphasize the following one.

\begin{theorem}~\cite{MCFPO}\label{Maltsev char unital} A finitely complete category $\C$ is a Mal'tsev category if and only if $\cod$ is unital.
\end{theorem}

Unitality of $\cod$ means that the category $\Pt_Y(\C)$ of points in $\C$ over $Y$ is a unital category, for every object $Y$ of $\C$. Recall from~\cite{MCFPO} that a pointed and finitely complete category is called a \emph{unital category} when, for all objects $A,C$ of $\C$, the pair of morphisms $(\la 1_A,0\ra \colon A\to A\times C, \;\la 0,1_C\ra \colon C\to A\times C)$ is jointly strongly epimorphic. The fact that $\Pt_Y(\C)$ is unital means: for every pullback of points over $Y$
\begin{equation}
\label{pb of split epis} \vcenter{\xymatrix@!0@=6em{ A\times_{Y}C
\splitsplitpullback \ar@<-.5ex>[d]_{\pi_A}
\ar@<-.5ex>[r]_(.7){\pi_C} & C \ar@<-.5ex>[d]_g
\ar@<-.5ex>[l]_-{\langle sg,1_C \rangle} \\
A \ar@<-.5ex>[u]_{\langle 1_A,tf \rangle} \ar@<-.5ex>[r]_f & Y
\ar@<-.5ex>[l]_s \ar@<-.5ex>[u]_t }}
\end{equation}
(which corresponds to a binary product in $\Pt_Y(\C)$), the pair of morphisms $(\langle 1_{A}, tf \rangle$, $\langle sg, 1_{C} \rangle )$ is jointly strongly epimorphic.

Since $\Pt_1(\C)\cong \C$, then every pointed Mal'tsev category $\C$ is necessarily unital. The converse is false. Indeed, $\Mon$ and $\SRng$ are examples of unital categories that are not Mal'tsev categories.

Inspired on the classification properties of the fibration of points $\cod$ studied in~\cite{MCFPO}, the authors of~\cite{MRVdL} explored several algebraic categorical notions, such as those of (strongly) unital~\cite{MCFPO}, subtractive~\cite{ZJanelidze-Subtractive}, Mal'tsev~\cite{CLP, CPP, CKP} and protomodular categories~\cite{Bourn1991}, at an object-wise level. This led to the corresponding notions of (strongly) unital, subtractive, Mal'tsev and protomodular objects. This approach allows one to distinguish ``good'' objects, i.e. with stronger algebraic properties, in a setting with weaker algebraic properties. The goal of this work was to obtain a categorical-algebraic characterisation of groups amongst monoids and of rings amongst semirings. It was shown in~\cite{MRVdL} that a monoid $Y$ is a group if and only if $Y$ is a \emph{Mal'tsev object} in $\Mon$ if and only if $Y$ is a \emph{protomodular object} in $\Mon$.

\begin{definition}\label{Maltsev obj} An object $Y$ of a finitely complete category $\C$ is called a \defn{Mal'tsev object} if the category $\Pt_Y(\C)$ of points over $Y$ is a unital category.
\end{definition}

It follows from Theorem~\ref{Maltsev char unital} that a finitely complete category $\C$ is a Mal'tsev category if and only all of its objects are Mal'tsev objects.

The notions of W-Mal'tsev object, studied in~\cite{Weighill}, and that of Mal'tsev object, in the sense of~\cite{MRVdL}, were obtained independently and developed with different goals in mind. Although their definitions are very different in nature, there is an obvious common property to both: a finitely complete category $\C$ is a Mal'tsev category if and only if all of its objects are W-Mal'tsev objects if and only if all of its objects are Mal'tsev objects. There are also several other properties shared by both notions, which we will recall in Section~\ref{W-Maltsev objs vs Maltsev objs}. Altogether, these observations led us to the natural question:
\begin{center}
	\textsf{To what extent are W-Mal'tsev and Mal'tsev objects comparable}?
\end{center}

\section{W-Mal'tsev object vs. Mal'tsev objects}\label{W-Maltsev objs vs Maltsev objs}
In this section we shall compare the notion of W-Mal'tsev object~\cite{Weighill} and that of Mal'tsev object~\cite{MRVdL} in a base category $\C$, which is finitely complete and may, eventually, admit some extra structure. We fix a finitely complete category $\C$ and denote by $\W(\C)$ (resp. $\M(\C)$), the full subcategory of $\C$ determined by the W-Mal'tsev (resp. Mal'tsev) objects of $\C$.

We begin by combining Theorems~\ref{Maltsev char difunctional} and~\ref{Maltsev char unital} into one:
\begin{theorem}\label{main Maltsev} Let $\C$ be a finitely complete category. The following statements are equivalent:
\begin{itemize}
	\item[(i)] $\C$ is a Mal'tsev category;
	\item[(ii)] diagram~\eqref{W-difunctional} is a difunctional relation in $\Set$, for any object $Y$ in $\C$ and any relation $X\stackrel{r_1}{\longleftarrow} R \stackrel{r_2}{\longrightarrow} Z$ in $\C$;
	\item[(iii)] for any pullback of points over an arbitrary object $Y$ in $\C$ as in~\eqref{pb of split epis}, the pair of morphisms $(\langle 1_{A}, tf \rangle$, $\langle sg, 1_{C} \rangle )$ is jointly strongly epimorphic.
\end{itemize}
\end{theorem}

An immediate consequence of Definitions~\ref{W-Maltsev obj} and~\ref{Maltsev obj} is the following one (see \textsf{(W1)} and \textsf{(M1)} below).
\begin{proposition} Let $\C$ be a finitely complete category. The following statements are equivalent:
\begin{itemize}
	\item[(i)] $\C$ is a Mal'tsev category;
	\item[(ii)] all objects of $\C$ are W-Mal'tsev objects;
	\item[(iii)] all objects of $\C$ are Mal'tsev objects.
\end{itemize}
\end{proposition}

The challenge of the comparison process is when the category $\C$ is not a Mal'tsev category, but admits (W-)Mal'tsev objects. Note that both definitions depend heavily on the surrounding category $\C$. To have a Mal'tsev object one must check a property for all pullbacks of points over that object; to have a W-Mal'tsev object one must check a property concerning all relations in $\C$. However, when $\C$ is a regular category~\cite{Barr} with binary coproducts, the result stated below in \textsf{(W5)} gives an independent (of all relations in $\C$) way to check that an object $Y$ is a W-Mal'tsev object. Consequently, one would expect that the notion of a Mal'tsev object is stronger than that of a W-Mal'tsev object. Indeed, that is the case as stated in Proposition~\ref{Maltsev obj => W-Maltsev obj}. Additionally, the greater demand on Mal'tsev objects may lead to ``trivial'' cases. This happens when $\C=\VCat$ (and $V$ is not a cartesian quantale), where the only Mal'tsev object is the empty set $\emptyset$, while a W-Mal'tsev object is precisely a symmetric $V_\wedge$-category (Theorem~\ref{thm for V-cats}). See also Example~\ref{Ex in Ordgrp} concerning (W-)Mal'tsev objects in $\OrdGrp$, the category of preordered groups. In other contexts, such as $\Mon$, the notions of W-Mal'tsev object and that of Mal'tsev object both coincide with a group (see Proposition~\ref{W-Maltsev objs in Mon are groups} and Theorem 6.14 in~\cite{MRVdL}).

In what follows it will be useful to consider properties on relations in a category using generalised elements (see~\cite{CKP}, for example). Let $R$ be a relation from $X$ to $Z$ given by the subobject $\la r_1,r_2\ra\colon R\to X\times Z$. Let $x\colon \colon A\to X$ and $z\colon A \to Z$ be morphisms, which can be considered as \emph{generalised elements} of $X$ and $Z$, respectively. We write $(x,z)\in_A R$ when the morphism $\la x,z\ra$ factors through $\la r_1,r_2\ra$
\[
	\xymatrix{A \ar[dr]_-{\la x,z \ra} \ar[rr] & & R \ar[dl]^-{\la r_1,r_2\ra} \\ & X\times Z.}
\]
Using this notation, an object $Y$ is a W-Mal'tsev object in $\C$ when: for any relation $\la r_1,r_2\ra\colon R\to X\times Z \in \C$ and morphisms $x_1,x_2\colon Y\to X$ and $z_1,z_2\colon Y \to Z$,
\begin{equation}\label{difunctionality}
	\left(\, (x_1,z_2)\in_Y R, (x_2,z_2)\in_Y R, (x_2,z_1)\in_Y R \,\right) \Rightarrow (x_1,z_1)\in_Y R.
\end{equation}

We extract from~\cite{Weighill} the main results which are used in the comparison process. Note that, in~\cite{Weighill} a Mal'tsev category need not be finitely complete by definition. Consequently, we adapted the result stated in \textsf{(W4)} to our finitely complete request.
\begin{itemize}
	\item[\textsf{(W1)}] $\W(\C)=\C$ if and only if $\C$ is a Mal'tsev category.
	\item[\textsf{(W2)}] $\W(\C)$ is closed under colimits and (regular) quotients in $\C$ (Proposition 2.1 in~\cite{Weighill}).
	\item[\textsf{(W3)}] When $\C$ is a well-powered regular category with coproducts, then $\W(\C)$ is a coreflective subcategory of $\C$; thus, $\W(\C)$ is (finitely) complete whenever $\C$ is (Corollary 2.2 in~\cite{Weighill}).
	\item[\textsf{(W4)}] Let $\C$ be a well-powered regular category with binary coproducts. Suppose that any relation in $\W(\C)$ is also a relation in $\C$. Then $\W(\C)$ is the largest full subcategory of $\C$ which is a Mal'tsev category and which is closed under binary coproducts and quotients in $\C$ (Corollary 2.5 in~\cite{Weighill}).
	\item[\textsf{(W5)}] Let $\C$ be a regular category with binary coproducts. An object $Y$ is a W-Mal'tsev object in $\C$ if and only if,
	given the (regular epimorphism, monomorphism) factorisation in $\C$
\begin{equation}\label{iotas}
\vcenter{\xymatrix{ 3Y \ar[dr]^e \ar[dd]_-{\left(\begin{smallmatrix} \iota_1 & \iota_2 \\ \iota_2 & \iota_2 \\ \iota_2 & \iota_1 \end{smallmatrix}\right)} \\ & R \ar[dl]^-{\la r_1,r_2\ra} \\ 2Y\times 2Y}}
\end{equation}
(which guarantees that $(\iota_1,\iota_2)\in_Y R$, $(\iota_2,\iota_2)\in_Y R$, $(\iota_2,\iota_1)\in_Y R$), we have $(\iota_1,\iota_1)\in_Y R$. As usual, we write $Y+Y=2Y$, $Y+Y+Y=3Y$, and $\iota_j\colon Y\to kY$ for the $j$-th coproduct coprojection (Proposition 2.3 in~\cite{Weighill}).
\end{itemize}

The main results we wish to emphasize from~\cite{MRVdL} are the following ones.
\begin{itemize}
	\item[\textsf{(M1)}] $\M(\C)=\C$ if and only if $\C$ is a Mal'tsev category (Proposition 6.10 in~\cite{MRVdL}).
	\item[\textsf{(M2)}] If $\M(\C)$ is closed under finite limits in $\C$, then $\M(\C)$ is a Mal'tsev category (Corollary 6.11 in~\cite{MRVdL}).
		\item[\textsf{(M3)}] When $\C$ is a pointed category, then $\C$ is a unital category if and only if $0$ is a Mal'tsev object (Proposition 6.4 in~\cite{MRVdL}).
	\item[\textsf{(M4)}] When $\C$ is a regular category, then $\M(\C)$ is closed under quotients in $\C$ (Proposition 6.12 in~\cite{MRVdL}).
	\item[\textsf{(M5)}] When $\C$ is a regular category, an object $Y$ in $\C$ is a Mal'tsev object if and only if every double split epimorphism over $Y$
\begin{equation}\label{double split epi}
\vcenter{\xymatrix@!0@=5em{ D \ar@<-.5ex>[d]_{g'} \ar@<-.5ex>[r]_{f'} & C
\ar@<-.5ex>[d]_g
\ar@<-.5ex>[l]_-{s'} \\
A \ar@<-.5ex>[u]_{t'} \ar@<-.5ex>[r]_f & Y \ar@<-.5ex>[l]_s
\ar@<-.5ex>[u]_t }}
\end{equation}
(meaning that the four ``obvious'' squares commute)  is a \emph{regular pushout}, i.e. the comparison morphism $\la g',f'\ra\colon D\to A\times_Y C$ is a regular epimorphism.
\end{itemize}

As a direct consequence of \textsf{(M3)}, we have a trivial comparison result in a pointed context.

\begin{proposition}\label{w-Maltsev obj does not imply Maltsev object} Let $\C$ be a pointed finitely complete category. Then the zero object is a W-Mal'tsev object, but it is not necessarily a Mal'tsev object.
\end{proposition}
\begin{proof} The zero object 0 is always a W-Mal'tsev object in any pointed category. Indeed, given any relation $R\to X\times Z$, the only generalised elements we can use are the zero morphisms $0_X\colon 0\to X$ and $0_Z\colon 0\to Z$. So, the implication in \eqref{difunctionality} obviously holds when $Y=0$. On the other hand, 0 is only a Mal'tsev object when $\C$ is a unital category by \textsf{(M3)}.
\end{proof}

When $\C$ is a regular category, then both $\W(\C)$ and $\M(\C)$ are closed under quotients in $\C$ -- \textsf{(M4)} and \textsf{(W2)}, where the later actually holds for any finitely complete $\C$. If $\C$ also has binary coproducts, we may use \textsf{(W5)} to conclude that the notion of Mal'tsev object is stronger than that of W-Mal'tsev object.

\begin{proposition}\label{Maltsev obj => W-Maltsev obj} Let $\C$ be a regular category with binary coproducts. If $Y$ is a Mal'tsev object in $\C$, then it is a W-Mal'tsev object in $\C$.
\end{proposition}
\begin{proof}
Consider the double split epimorphism over $Y$
\[
\xymatrix@!0@=8em{
	3Y \ar@<-.5ex>[d]_{\threesum{\iota_1}{\iota_2}{\iota_2}} \ar@<-.5ex>[r]_{\threesum{\iota_2}{\iota_2}{\iota_1}} & 2Y
\ar@<-.5ex>[d]_-{\nabla} \ar@<-.5ex>[l]_-{\iota_{32}=\twosum{\iota_3}{\iota_2}} \\
2Y \ar@<-.5ex>[u]_{\iota_{12}=\twosum{\iota_1}{\iota_2}} \ar@<-.5ex>[r]_{\nabla=\twosum{1_Y}{1_Y}} & Y, \ar@<-.5ex>[l]_{\iota_2}
\ar@<-.5ex>[u]_{\iota_2} }
\]
which is a regular pushout by assumption. We then get a (regular epimorphism, monomorphism) factorisation as in \eqref{iotas}, where $R=\Eq(\nabla)$ is the kernel pair of $\nabla$.
It easily follows that $(\iota_1,\iota_1)\in_Y R$, since $R=\Eq(\nabla)$ and  $\nabla\iota_1=\nabla\iota_1$. This shows that $Y$ is a W-Mal'tsev object by \textsf{(W5)}.
\end{proof}

An interesting question is whether $\W(\C)$ and $\M(\C)$ are themselves Mal'tsev categories. From \textsf{(W3)} we know that, if $\C$ is a well-powered regular category with coproducts, then $\W(\C)$ is finitely complete. Also, if every relation in $\W(\C)$ is also a relation in $\C$, then the full subcategory $\W(\C)$ is a Mal'tsev category; actually, it is the largest one which is closed under binary coproducts and quotients in $\C$ (see \textsf{(W4)}). On the other hand, there are no ``obvious'' conditions on $\C$ from which we could deduce finite completeness for $\M(\C)$. So, we can only conclude that $\M(\C)$ is a Mal'tsev category when it is closed under finite limits in $\C$; this is \textsf{(M2)}. When this is the case, we still do not know whether $\M(\C)$ is the largest full subcategory of $\C$ which is a Mal'tsev category.

In~\cite{MRVdL} is was shown that the Mal'tsev objects in $\Mon$ are precisely the groups (Theorem 6.14). This is also the case with respect to W-Mal'tsev objects in $\Mon$. It follows from the above and the next proposition that $\Grp$ is the largest full subcategory of $\Mon$ which is a Mal'tsev category and is closed under binary coproducts and quotients in $\Mon$.

\begin{proposition}\label{W-Maltsev objs in Mon are groups}
A monoid $Y$ is a W-Mal'tsev object in $\Mon$ if and only if $Y$ is a group.
\end{proposition}
\begin{proof}
If $Y$ is a group, then it is a Mal'tsev object in $\Mon$; thus, it is a W-Mal'tsev object in $\Mon$ by Proposition~\ref{Maltsev obj => W-Maltsev obj}. For the converse, suppose that $(Y,+,0)$ is a W-Mal'tsev object in $\Mon$. We use additive notation although $Y$ is not necessarily an abelian monoid. Since $\Mon$ is a regular category with binary coproducts, we can apply \textsf{(W5)} to conclude that $(\iota_1,\iota_1)\in_Y R$, where $R$ is as in diagram~\eqref{iotas}. So, for any $x\in Y$, $x\neq 0$, we have $([\underline{x}], [\underline{x}])\in R$. (We use the notations $\iota_1(x)=[\,\underline{x}\,]$, $\iota_2(x)=[\,\overline{x}\,]$, for any $x\in Y$, for the coprojections $\iota_1,\iota_2\colon Y\to 2Y$, and $\iota_1(x)=[\,\underline{x}\,]$, $\iota_2(x)=[\,\overline{x}\,]$, $\iota_3(x)=[\widetilde{x}]$, for any $x\in Y$, for the coprojections $\iota_1,\iota_2, \iota_3\colon Y\to 3Y$.) Since $e\colon 3Y\to R$ is surjective, there exists an element $[\,\underline{u_1}\overline{v_1}\widetilde{w_1}\cdots \underline{u_k}\overline{v_k}\widetilde{w_k}\,]\in 3Y$ such that $e([\,\underline{u_1}\overline{v_1}\widetilde{w_1}\cdots \underline{u_k}\overline{v_k}\widetilde{w_k}\,])=([\underline{x}], [\underline{x}])$.  We deduce that
\begin{center} 
$\left\{ \begin{array}{l} \vspace{20pt} \end{array}\right.$\hspace{-10pt}
\begin{tabular}{l}
	$[\, \underline{u_1}\overline{v_1+w_1}\cdots \underline{u_k}\overline{v_k+w_k}\,]= [\underline{x}]$ \vspace{3pt}\\
	$[\,\overline{u_1+v_1}\underline{w_1}\cdots \overline{u_k+v_k}\underline{w_k}\,] = [\underline{x}]$;
\end{tabular}
\end{center}
consequently
\[
\left\{
\begin{array}{l}
	v_1+w_1=0, \cdots, v_k+w_k=0,\;\;\mathrm{and}\;\; u_1+\cdots+ u_k=x  \vspace{3pt}\\
	u_1+v_1=0, \cdots, u_k+v_k=0,\;\;\mathrm{and}\;\; w_1+\cdots+w_k=x.
\end{array}\right.
\]
We may define the element $y=v_k+\cdots + v_1$ of $Y$ which is the inverse of $x$:
\[
\begin{array}{l}
 x+y = u_1+\cdots + u_{k-1} + (u_k + v_k) + v_{k-1} \cdots + v_1 = u_1+\cdots + (u_{k-1} + v_{k-1}) \cdots + v_1  = \cdots = 0,\\
 y+x=v_k+\cdots +v_2 + (v_1+w_1)+ w_2 +\cdots + w_k=v_k+\cdots + (v_2 + w_2) +\cdots + w_k = \cdots =0.
\end{array}
\]
\end{proof}

We finish this section with the example of (W-)Mal'tsev objects in $\OrdGrp$, the category of preordered groups. We denote a preordered group by $(Y,+,\le)$, even though the associated group $(Y,+,0)$ is not necessarily abelian. Recall from~\cite{CM-FM} that the \emph{positive cone} of $Y$, $P_Y=\{x\in Y: 0\le x\}$, is always a submonoid of $Y$ which is closed under conjugation in $Y$. It was shown in~\cite{CM-FM} that a preordered group $(Y,+,\le)$ is a Mal'tsev object if and only if the preorder relation $\le$ is an equivalence relation if and only if $P_Y$ is a group. We follow a similar argument as that of the proof of Proposition~\ref{W-Maltsev objs in Mon are groups} to analyse, in the next example, properties of the positive cone of a W-Mal'tsev object in $\OrdGrp$.

\begin{example}\label{Ex in Ordgrp}
Let  $(Y,+,\le)$ be a W-Mal'tsev object in $\OrdGrp$. Since $\OrdGrp$ is a regular category with binary coproducts, we can apply \textsf{(W5)} to conclude that $(\iota_1,\iota_1)\in_Y R$, where $R$ is as in diagram~\eqref{iotas}. So, for any $x\in P_Y$, $x\neq 0$, we have $([\underline{x}], [\underline{x}])\in P_R$.  Since $e$ is a regular epimorphism, there exists a positive element $[\,\underline{u_1}\overline{v_1}\widetilde{w_1}\cdots \underline{u_k}\overline{v_k}\widetilde{w_k}\,]\in P_{3Y}$ such that $e([\,\underline{u_1}\overline{v_1}\widetilde{w_1}\cdots \underline{u_k}\overline{v_k}\widetilde{w_k}\,])=([\underline{x}], [\underline{x}])$.  We deduce that
\begin{center} 
$\left\{ \begin{array}{l} \vspace{20pt} \end{array}\right.$\hspace{-10pt}
\begin{tabular}{l}
	$[\, \underline{u_1}\overline{v_1+w_1}\cdots \underline{u_k}\overline{v_k+w_k}\,]= [\underline{x}]$ \vspace{3pt}\\
	$[\,\overline{u_1+v_1}\underline{w_1}\cdots \overline{u_k+v_k}\underline{w_k}\,] = [\underline{x}]$;
\end{tabular}
\end{center}
consequently
\[
\left\{
\begin{array}{l}
	v_1+w_1=0, \cdots, v_k+w_k=0,\;\;\mathrm{and}\;\; u_1+\cdots + u_k=x  \vspace{3pt}\\
	u_1+v_1=0, \cdots, u_k+v_k=0,\;\;\mathrm{and}\;\; w_1+\cdots+w_k=x.
\end{array}\right.
\]
We conclude that $u_1=w_1, \cdots, u_k=w_k$ and the above positive element of $3Y$ has the shape $[\,\underline{u_1}\overline{\mbox{-}u_1}\widetilde{u_1}\cdots \underline{u_k}\overline{\mbox{-}u_k}\widetilde{u_k}\,]$. The monotone morphism
\[
	\begin{array}{rccc}
		\threesum{0}{1_Y}{0} \colon & 3Y & \to & Y \\
		& [\,\underline{u_1}\overline{\mbox{-}u_1}\widetilde{u_1}\cdots \underline{u_k}\overline{\mbox{-}u_k}\widetilde{u_k}\,] & \mapsto & -u_1 - \cdots - u_k
	\end{array}
\]
takes positive elements of $3Y$ to positive elements of $Y$; thus $-u_1-\cdots - u_k\in P_Y$.

If $Y$ is an abelian group, then from $x=u_1+\cdots + u_k\in P_Y$, we deduce $-x=-u_1-\cdots - u_k\in P_Y$, i.e. $P_Y$ is a group. In this case W-Mal'tsev objects and Mal'tsev objects coincide in $\OrdGrp$. We do not know whether the two notions coincide for the non-abelian case.
\end{example}

\section{$V$-categories and (W-)Mal'tsev objects}

In this section we will generalise the characterisations of W-Mal'tsev objects in the duals of the category of metric spaces and of the category of topological spaces obtained in \cite{Weighill}. For that we will make use of the concepts of $V$-category and of $(U,V)$-category. Here, as in \cite{Law}, the notion of $V$-category will play the role of metric space, while a $(U,V)$-category (as introduced in \cite{CT03}) will play the role of a topological space. We start by presenting the basic tools of this approach.

Throughout $V$ is a unital and integral quantale; that is, $V$ is a complete lattice equipped with a tensor product $\otimes$, with unit $k=\top\neq\bot$, that distributes over arbitrary joins. As a category, $V$ is a monoidal closed category.

When more than one tensor product may be considered, we use the notation $V_\otimes$ to indicate that we are using the tensor product $\otimes$ in $V$.

\begin{definition}
A \defn{$\mathbf{\emph{V}}$-category} is a set $X$ together with a map $X\times X\to V$, whose image of $(x,x')$ we denote by $X(x,x')$, such that, for each $x,x',x''\in X$,
\begin{enumerate}
\item[(R)] $k\leq X(x,x)$;
\item[(T)] $X(x,x')\otimes X(x',x'')\leq X(x,x'')$.
\end{enumerate}
A $V$-category is said to be \emph{symmetric} if, for every $x,x'\in X$, $X(x,x')=X(x',x)$.\\
A \emph{$V$-functor} $f\colon X\to Y$ is a map such that, for all $x,x'\in X$, $X(x,x')\leq Y(f(x),f(x'))$.
\end{definition}

The two axioms of a $V$-category express the existence of identities and the categorical composition law, but they may also be seen as a reflexivity and a transitivity condition, or even as two conditions usually imposed to metric structures: if in the complete half real line $[0,\infty]$ we consider the order relation $\geq$ and the tensor product $+$, (R) and (T) read as
\begin{enumerate}
\item[(R')] $0\geq X(x,x)$;
\item[(T')] $X(x,x')+X(x',x'')\geq X(x,x'')$,
\end{enumerate}
with (R) meaning that the distance from a point to itself is $0$, and (T) the usual triangular inequality.

We denote by $\VCat$ the category of $V$-categories and $V$-functors, and by $\VCat_\sym$ its full subcategory of symmetric $V$-categories.

\begin{remark}
When $V$ is a complete lattice which is a \emph{frame}, so that finite meets distribute over arbitrary joins, then $V_\wedge=(V,\leq,\wedge,\top)$ is a (unital and integral) quantale. We call such quantales \emph{cartesian}. It is well-known that, if $V_\wedge$ is a cartesian quantale, then the category $V_\wedge\mbox{-}\Cat$ has special features, like being cartesian closed~\cite{Law}. Here we will show that it also has a key role in the study of W-Mal'tsev objects of $(V_\wedge\mbox{-}\Cat)^\op$.

Moreover, even if the quantale $V_\otimes$ is not a frame, the $V_\otimes$-categories $X$ which, in addition, verify
\[X(x,x')\wedge X(x',x'')\leq X(x,x'')\]
will be specially relevant, and we will call them also $V_\wedge$-categories. We note that, since $k=\top$, they also satisfy $\top\leq X(x,x)$ for every $x\in X$; moreover, $u\otimes v\leq u\wedge v$ always holds, so the above condition guarantees immediately condition (T) of the above definition.
\end{remark}

\begin{examples}
\begin{enumerate}
\item When $V=2=(\{0<1\},\wedge,1)$, a $V$-category is a preordered set and a $V$-functor is a monotone map. Hence $\VCat$ is the well-known category $\Ord$ of preordered sets and monotone maps.

\item When $V=[0,\infty]_+=([0,\infty],\geq,+,0)$, that is $V$ is the complete lattice $[0,\infty]$, ordered by $\geq$, with tensor product $\otimes=+$, a $V$-category is a (\emph{generalised}) \emph{Lawvere metric space} \cite{Law} (not necessarily separated nor symmetric, with $\infty$ as a possible distance), since the two conditions above mean that $X(x,x)=0$ and $X(x,y)+X(y,z)\geq X(x,z)$, for $x,y,z\in X$, and a $V$-functor is a \emph{non-expansive map}.\\
    When $V$ is the cartesian quantale $[0,\infty]_\max=([0,\infty],\geq,\max,0)$, with tensor product $\otimes=\max$, then $V_\max\mbox{-}\Cat$ is the category of (\emph{generalised}) \emph{ultrametric spaces}.

\item The complete lattice $([0,1],\leq)$ can be equipped with several tensor products -- usually called \emph{t-norms} -- including the \emph{Lukasiewicz sum}, which lead to interesting instances of categories of the form $\VCat$, like Lawvere metric spaces, ultrametric spaces, and \emph{bounded metric spaces}.

\item The set of \emph{distribution functions}
\[\Delta=\{\varphi\colon[0,\infty]\to[0,1]\,;\mbox{ for all }\alpha\in[0,\infty]\; \;\varphi(\alpha)=\bigvee_{\beta<\alpha}\,\varphi(\beta)\]
with the pointwise order is a complete lattice. The tensor product is defined, for each $\varphi,\psi\in\Delta$, by
\[(\varphi\otimes\psi)(\alpha)=\bigvee_{\beta+\gamma\leq\alpha}\,\varphi(\beta)\times\psi(\gamma),\]
having as unit the distribution function $\kappa\colon[0,\infty]\to[0,1]$, with $\kappa(\alpha)=0$ if $\alpha=0$ and $\kappa(\alpha)=1$ otherwise. Then $\Delta$-$\Cat$ is the category of \emph{probabilistic metric spaces} and (\emph{probabilistic}) \emph{non-expansive} maps.
\end{enumerate}
For more details and examples see for instance \cite{CH17}.
\end{examples}

The forgetful functor $\VCat\to\Set$ is topological, hence $\VCat$ is complete and cocomplete, with limits and colimits formed as in $\Set$ and equipped with the corresponding initial and final $V$-category structures, respectively. In \cite[Corollary 8]{MST} it is shown that, under suitable conditions, $\VCat$ is an extensive category. In $\VCat$ epimorphisms are pullback-stable, since they are exactly surjective $V$-functors and pullbacks are formed as in $\Set$. Therefore, as shown in \cite[Proposition 3]{MF}, \emph{$(\VCat)^\op$ is a weakly Mal'tsev category}. In addition, in \cite[Theorem 4.6]{HN23} it is shown that \emph{$(\VCat)^\op$ is a quasi-variety}, so in particular it is a regular category.
Still, we think it is worth to prove here directly that \emph{$(\VCat)^\op$ is a regular category}. Indeed, in $\VCat$ regular monomorphisms coincide with extremal monomorphisms, and are exactly the injective maps $f\colon X\to Y$ such that $X(x,x')=Y(f(x),f(x'))$, for all $x,x'\in X$. Moreover, $\VCat$ has the stable orthogonal factorisation system (epimorphism, regular monomorphism), which factors every $V$-functor $f\colon X\to Y$ as
\[\xymatrix{X\ar[rr]^f\ar[rd]_{e}&&Y\\
&Z=f(X)\ar[ru]_m&}\]
where $e$ is the corestriction of $f$ to $Z=f(X)$ and $Z(y,y')=Y(y,y')$ for all $y,y'\in Z$. Given a pushout in $\VCat$
\[\xymatrix{X\ar[r]^m\ar[d]_f&Z\ar[d]^g\\
Y\ar[r]_n&W}\] with $m$ a regular monomorphism, for simplicity we assume that $m$ is an inclusion. We know that, since pushouts are formed as in $\Set$, we may consider $W=(Y+Z)/\sim$, where, for $y\in Y$ and $z\in Z$, $y\sim z$ exactly when $z\in X$ and $f(z)=y$, and $n$ an inclusion. Then, for every $y,y'\in Y$,
\[\begin{array}{rcll}
W([y],[y'])&=&\displaystyle Y(y,y')\vee\bigvee_{x\sim y,\,x'\sim y'}\,Z(x,x')\\
&=&\displaystyle Y(y,y')\vee\bigvee_{x\sim y,\,x'\sim y'}\,X(x,x')&\mbox{(because $m$ is a regular monomorphism)}\\
&=&Y(y,y')&\mbox{(because $f$ is a $V$-functor)}
\end{array}\]
and therefore the inclusion $n\colon Y\to W$ is a regular monomorphism.

Hence, since $(\VCat)^\op$ is a regular category with binary coproducts we know by \textsf{(W5)} that a $V$-category $Y$ is a W-Mal'tsev object if, and only if, given the $V$-functor $f\colon Y^2+Y^2\to Y^3$ defined by
\[\xymatrix{Y^2\ar@/^.7pc/[rrd]^{\la\pi_1,\pi_2,\pi_2\ra}\ar[rd]_{\iota_1}\\
&Y^2+Y^2\ar[r]|-{\,f\,}&Y^3\\
Y^2\ar[ru]^{\iota_2}\ar@/_.7pc/[rru]_{\la\pi_2,\pi_2,\pi_1\ra}}\]
where $\pi_i\colon Y^2\to Y$, $i=1,2$, are the product projections, and its (epimorphism, regular mo\-no\-mor\-phism)-factorisation in $\VCat$
\begin{equation}\label{eq:fact}
(\xymatrix{Y^2+Y^2\ar[r]^-f&Y^3})=(\xymatrix{Y^2+Y^2\ar[r]^-e&X\ar[r]^-m&Y^3})\end{equation}
there is a unique $V$-functor $g\colon X\to Y$ such that the following diagram commutes
\begin{equation}\label{eq:g}\xymatrix{Y^2\ar@/^.7pc/[rrrd]^{\pi_1}\ar[rd]_{\iota_1}\\
&Y^2+Y^2\ar[r]|-{\,e\,}&X\ar[r]|{\,g\,}&Y\\
Y^2\ar[ru]^{\iota_2}\ar@/_.7pc/[rrru]_{\pi_1}}\end{equation}
As a set, $X$ is the image of the map $f$, that is \[X=X_1\cup X_2, \mbox{ with }X_1=\{(x,y,y)\,;\,x,y\in Y\},\;X_2=\{(x,x,y)\,;\,x,y\in Y\};\]
its $V$-category structure is inherited from $Y^3$, that is
\[X((x,y,z),(x',y',z'))=Y(x,x')\wedge Y(y,y')\wedge Y(z,z').\]
The map $g\colon X\to Y$ must assign $x$ both to each $(x,y,y)$ and each $(y,y,x)$.

\begin{theorem}\label{thm for V-cats}
Let $Y$ be a $V$-category.
\begin{enumerate}
\item The following conditions are equivalent:
\begin{enumerate}
\item[(i)] $Y$ is a W-Mal'tsev object in $(\VCat)^\op$;
\item[(ii)] for all $x,y,z\in Y$, $Y(z,x)\wedge Y(z,y)\leq Y(x,y)$;
\item[(iii)] $Y$ is a symmetric $V_\wedge$-category.
\end{enumerate}
\item If $V$ is not a cartesian quantale, then $Y$ is a Mal'tsev object in $(\VCat)^\op$ if, and only if, $Y=\emptyset$.
\end{enumerate}
\end{theorem}

\begin{proof}
(1)
(i) $\Rightarrow$ (ii): If the map $g\colon X\to Y$ from \eqref{eq:g} is a $V$-functor, then in particular
\[X((x,z,z),(x,x,y))=Y(x,x)\wedge Y(z,x)\wedge Y(z,y)\leq Y(g(x,z,z),g(x,x,y))=Y(x,y).\]
(ii) $\Rightarrow$ (iii): With $z=y$ in (ii) one gets
\[Y(y,x)=Y(y,x)\wedge Y(y,y)\leq Y(x,y),\]
hence $Y$ is symmetric. Then, for any $x,y,z\in Y$,
\[Y(x,z)\wedge Y(z,y)=Y(z,x)\wedge Y(z,y)\leq Y(x,y).\]
(iii) $\Rightarrow$ (i): We want to prove that $g$ is a $V$-functor; that is, for each $(x,y,z), (x',y',z')$ in $X$,
\[X((x,y,z),(x',y',z'))=Y(x,x')\wedge Y(y,y')\wedge Y(z,z')\leq Y(g(x,y,z),(x',y',z')).\]
In case $x=y$ and $x'=y'$, or in case $y=z$ and $y'=z'$, the inequality is trivially satisfied; when $x=y$ and $y'=z'$, using (iii), one obtains
\[Y(x,x')\wedge Y(x,y')\wedge Y(z,y')=Y(z,y')\wedge Y(y',x)\wedge Y(x,x')\leq Y(z,x'),\]
and in case $y=z$ and $x'=y'$ one gets
\[Y(x,x')\wedge Y(y,x')\wedge Y(y,z')=Y(x,x')\wedge Y(x',y)\wedge Y(y,z')\leq Y(x,z').\]

\noindent (2) Since $(\VCat)^\op$ is regular we may use the characterisation of Mal'tsev object of \textsf{(M5)}. First we point out that a double split epimorphism \eqref{double split epi} in $(\VCat)^\op$ is also a double split epimorphism in $\VCat$, with the role of split epimorphisms and split monomorphisms interchanged:
\begin{equation}\label{eq:op}
\vcenter{\xymatrix@!0@=5em{ D \ar@<.5ex>[d]^{t'} \ar@<.5ex>[r]^{s'} & C
\ar@<.5ex>[d]^t
\ar@<.5ex>[l]^-{f'} \\
A \ar@<.5ex>[u]^{g'} \ar@<.5ex>[r]^s & Y. \ar@<.5ex>[l]^f
\ar@<.5ex>[u]^g }}
\end{equation}
(this diagram is just the dual of \eqref{double split epi}).

Let $V$ be a non-cartesian quantale. To show that the only Mal'tsev object in $(\VCat)^\op$ is the empty set it is enough to check that $Y=\{1\}$, with $Y(1,1)=k$, is not a Mal'tsev object, since by (M4) Mal'tsev objects are closed under quotients and, for any non-empty $V$-category $Z$, any map $Y\to Z$ is a split monomorphism in $\VCat$.

We will build a double split epimorphism in $\VCat$ as above
 so that it is not a \emph{regular pullback}, i.e. the comparison morphism $\langle g',f'\rangle\colon A+_YC\to D$ is not a regular monomorphism. If $V$ is not cartesian then there exist $u,v\in V$ such that $u\otimes v< u\wedge v$.
 Consider in \eqref{eq:op} the symmetric $V$-categories $A=\{0,1\}$, $C=\{1,2\}$, and $D=\{0,1,2\}$, with $A(0,1)=u$, $C(1,2)=v$, and $D(0,2)=u\wedge v$, with the obvious $V$-functors (with $s'(0)=1$ and $t'(2)=1$). Then \eqref{eq:op} is a double split epimorphism but the comparison morphism $\langle g',f'\rangle\colon A+_YC\to D$, which is in fact a bijection, is not a regular monomorphism because in the pushout $(A+_YC)(0,2)$ must be $u\otimes v$.
\end{proof}

\begin{remark}\label{re:R1}
We point out that condition (ii) in case $\otimes=\wedge$ is the separation axiom (R1) studied in \cite[Section 2.2]{CCT}.
\end{remark}

The two statements of Theorem~\ref{thm for V-cats} have an immediate consequence, showing how drastically different may be the two notions of Mal'tsev object in the context of $(\VCat)^\op$. For instance, if $V=[0,\infty]_+$, then a Lawvere metric space is a W-Mal'tsev object exactly when it is a symmetric ultrametric space, while it is a Mal'tsev object only if it is empty.

\begin{corollary}
\begin{enumerate}
\item If $V$ is a cartesian quantale, then $(\VCat_\sym)^\op$ is a Mal'tsev category.
\item For every unital and integral quantale $V$, the largest Mal'tsev subcategory of $(\VCat)^\op$ closed under binary coproducts and regular epimorphisms is $(V_\wedge$-$\Cat_\sym)^\op$.
\end{enumerate}
\end{corollary}

It is straightforward to check that, if we restrict our study to symmetric $V_\otimes$-categories, in case the tensor product $\otimes$ is commutative the double split epimorphism built above also works. Hence we may also conclude that Mal'tsev objects trivialise is this more restrictive setting.

\begin{corollary}
If $V$ is a commutative unital and integral non-cartesian quantale, then the only Mal'tsev object in $(\VCat_\sym)^\op$ is the empty $V$-category.
\end{corollary}

In order to generalise Weighill's result for topological spaces we need to introduce an extra ingredient, the \emph{ultrafilter monad} on $\Set$, and its extension to $V\mbox{-}\Rel$ (see \cite{CT03} for details). As Barr showed in \cite{Barr70}, a topological space can be described via its ultrafilter convergence, meaning that a topological space can be given by a set $X$, together with a relation between ultrafilters on $X$ and points of $X$, $UX\relto X$ such that, for every $x\in X$, $\xx\in UX$, and $\xxx\in U^2X$,
\begin{enumerate}
\item $\ultra{x}\,\to x$ (where $\ultra{x}$ is the principal ultrafilter defined by $x$);
\item $\xxx\to\xx$ and $\xx\to x$ $\implies$ $\mu(\xxx)\to x$ (where $\mu(\xxx)$ is the Kowalski sum of $\xxx$).
\end{enumerate}
Given topological spaces $X$ and $Y$, a map $f\colon X\to Y$ is continuous if, for all $\xx\in UX$ and $x\in X$,  $Uf(\xx)\to f(x)$ whenever $\xx\to x$.

This approach can be generalised making use of the ultrafilter monad $U$, a unital and integral quantale $V$, and a lax extension of $U$ to $V\mbox{-}\Rel$.

\begin{definition}
A \defn{$\mathbf{(\emph{U,V})}$-category} is a set $X$ together with a map $a\colon UX\times X\to V$ whose image of $(\xx,x)$ we denote by $X(\xx,x)$, such that, for each $x\in X$, $\xx\in UX$, $\xxx\in U^2X$,
\begin{enumerate}
\item[(R)] $k\leq X(\ultra{x},x)$;
\item[(T)] $X(\xxx,\xx)\otimes X(\xx,x)\leq X(\mu(\xxx),x)$,
\end{enumerate}
where by $X(\xxx,\xx)$ we mean the image, under $U$, of the $V$-relation $a\colon UX\relto X$.\\
A $(U,V)$-functor $f\colon X\to Y$ is a map such that $X(\xx,x)\leq Y(Uf(\xx),f(x))$, for all $x\in X$, $\xx\in UX$.
\end{definition}
We denote by $(U,V)\mbox{-}\Cat$ the category of $(U,V)$-categories and $(U,V)$-functors.

\begin{examples}
\begin{enumerate}
\item As shown by Barr \cite{Barr70}, if $V=2$, then $(U,V)\mbox{-}\Cat$ is isomorphic to the category $\Top$ of topological spaces and continuous maps.
\item As shown in \cite{CH03}, if $V=[0,\infty]_+$ $(U,V)\mbox{-}\Cat$ is isomorphic to the category $\App$ of Lowen's \emph{approach spaces} and non-expansive maps.
\end{enumerate}
\end{examples}

\begin{proposition}
The category $((U,V)\mbox{-}\Cat)^\op$ is regular.
\end{proposition}
\begin{proof}
As in $\VCat$, in $(U,V)\mbox{-}\Cat$ (epimorphisms, regular monomorphisms) form a stable factorisation system, and a morphism $m\colon X\to Z$ is a regular monomorphism if, and only if, it is an injective map and $X(\xx,x)=Z(Um(\xx),m(x))$, for all $\xx\in UX$ and $x\in X$.
Given a pushout in $(U,V)\mbox{-}\Cat$
\[\xymatrix{X\ar[r]^m\ar[d]_f&Z\ar[d]^g\\
Y\ar[r]_n&W}\] with $m$ a regular monomorphism, which we assume, for simplicity, to be an inclusion, as for $\VCat$ we may consider $W=(Y+Z)/\sim$, where, for $y\in Y$ and $z\in Z$, $y\sim z$ exactly when $z\in X$ and $f(z)=y$. Then $n$, and therefore also $Un$, is an injective map, and, for every $\yy\in UY$ and $y\in Y$,
\[W(Un(\yy),[y])=\displaystyle Y(\yy,y)\vee\bigvee_{Ug(\zz)=Un(\yy),\,x\sim y}\,Z(\zz,x).\]
Any $\zz\in UZ$ with $Ug(\zz)=Un(\yy)$ must contain $X$, and consequently it is the image under $Um$ of an ultrafilter $\xx$ in $X$; hence $Z(\zz,x)=X(\xx,x)\leq Y(Uf(\xx),f(x))=Y(\yy,y)$. This implies that $Y(\yy,y)=W(Un(\yy),[y])$,
and therefore $n\colon Y\to W$ is a regular monomorphism.
\end{proof}

\begin{proposition}
\begin{enumerate}
\item For a $(U,V)$-category $Y$ the following conditions are equivalent:
\begin{enumerate}
\item[(i)] $Y$ is a W-Mal'tsev object in $((U,V)\mbox{-}\Cat)^\op$;
\item[(ii)] For all $\zz\in UY$, $x,y\in Y$,
\begin{enumerate}
\item[(a)] $Y(\zz,x)\wedge Y(\zz,y)\leq Y(\ultra{x},y)$;
\item[(b)] $Y(\zz,x)\wedge Y(\ultra{x},y)\leq Y(\zz,y)$.
\end{enumerate}
\end{enumerate}
\item If $V$ is not a cartesian quantale, then $Y$ is  Mal'tsev object in  $((U,V)\mbox{-}\Cat)^\op$ if, and only if, $Y=\emptyset$.
\end{enumerate}
\end{proposition}

\begin{proof}
(1) The map $g\colon X\to Y$ is a $(U,V)$-functor if, and only if, for every $\ww\in UX$ and $(x,y,z)\in X$,
\[
	\resizebox{\textwidth}{!}{$X(\ww,(x,y,z))=Y(U\pi_1(\ww),x)\wedge Y(U\pi_2(\ww),y)\wedge Y(U\pi_3(\ww),z)\leq Y(Ug(\ww),g(x,y,z)).$}
\]
For simplicity, we will denote the $i$-th projection $Y^k\to Y$ ($k=2,3$) by $\pi_i$ and $U\pi_i(\ww)$ by $\ww_i$.

Let $h\colon Y\times Y\to X$ with $h(x,y)=h(x,y,y)$.

(i) $\Rightarrow$ (ii): Fix $\zz\in UY$ and $x\in Y$. Since the canonical map $U(Y^2)\to U(Y)^2$ is surjective, there exists $\widetilde{\zz}\in U(Y^2)$ such that $\widetilde{z_1}=\ultra{x}$ and $\widetilde{z_2}=\zz$. Then $\ww=Uh(\widetilde{\zz})\in UX$ and, since $g\cdot h=\pi_1$, $\ww_1=\ultra{x}$, $\ww_2=\zz=\ww_3$, and $Ug(\ww)=\ultra{x}$. Then we have
\[X(\ww,(x,x,y))\leq Y(\ultra{x},y)\;\Rightarrow\; Y(\ultra{x},x)\wedge Y(\zz,x)\wedge Y(\zz,y)\leq Y(\ultra{x},y),\]
that is (a) holds.\\
Analogously, let $\widehat{\zz}\in U(Y^2)$ be such that $\widehat{z_1}=\zz$ and $\widehat{z_2}=\ultra{x}$, and let $\xx=Uh(\widehat{\zz})$, so that $Ug(\xx)=\zz$. Then, for every $y\in Y$,
\[X(\xx,(x,x,y))\leq Y(\zz,y)\;\Rightarrow\; Y(\zz,x)\wedge Y(\ultra{x},x)\wedge Y(\ultra{x},y)\leq Y(\zz,y),\]
that is (b) holds.

(ii) $\Rightarrow$ (i): Conversely, assume (a) and (b). Let $\ww\in UX$ and assume that $X_1\in\ww$ (for the case $X_2\in\ww$ the proof is analogous). Then $\ww=Uh(\yy)$ for some $\yy\in U(Y^2)$. Then
\[X(\ww,(x,y,y))=Y(\yy_1,x)\wedge Y(\yy_2,y)\wedge Y(\yy_2,y)\leq Y(\yy_1,x)=Y(Ug(\ww),x),\]
\[X(\ww,(x,x,y))=Y(\yy_1,x)\wedge Y(\yy_2,x)\wedge Y(\yy_2,y)\leq Y(\yy_1,x)\wedge Y(\ultra{x},y)\leq Y(\yy_1,y);\]
that is, $g$ is a $(U,V)$-functor.

\noindent (2) We may argue as in the proof of Theorem \ref{thm for V-cats}, and use $(U,V)$-categories $Y=\{1\}$, $A=\{0,1\}$, $C=\{1,2\}$ and $D=\{0,1,2\}$ with the structures given by $A(\ultra{0},1)=A(\ultra{1},0)=u$, $C(\ultra{1},2)=C(\ultra{2},1)=v$, and $D(\ultra{0},2)=D(\ultra{2},0)=u\otimes v$.
\end{proof}

\begin{remark}
The result of the proposition above applies immediately to $\Top$ and $\App$, and it may be extended to other suitable $\Set$-monads.
For simplicity we chose to restrict ourselves to the ultrafilter monad, recovering the result in $\Top$ and obtaining a characterisation of W-Mal'tsev objects in $(\App)^\op$.

At a first glance it may seem that here we need an extra condition to recover the result in $\Top$. Indeed, in $(\Top)^\op$ condition (a) is the well-known separation axiom (R1) we have mentioned in Remark \ref{re:R1}, while condition (b) follows from axiom (T) of the definition of $(U,V)$-category in case $V$ is cartesian, which is the case of $\Top$. When $V=[0,\infty]_+$ condition (a) is stronger than (R1), which reads as, for every $\zz\in UY$, $x,y\in Y$, $Y(\zz,x)+Y(\zz,y)\geq Y(\ultra{x},y)$, and (b) does not follow from axiom (T).
\end{remark}



\begin{thebibliography}{10}
\bibitem{Barr70} M.~Barr, \emph{Relational algebras}, in: Reports of the Midwest Category
  Seminar, IV, Lecture Notes in Mathematics, Vol. \textbf{137}. Springer, Berlin, 1970,
  pp. 39--55.
\bibitem{Barr} M.~Barr, P.A.~Grillet and D.H.~van~Osdol, \emph{Exact categories and categories of sheaves}, Lecture Notes in Math. 236, Springer-Verlag (1971).
\bibitem{BB} F.~Borceux and D.~Bourn, \emph{Mal'cev, protomodular, homological and semi-abelian categories}, Math. Appl., vol. 566, Kluwer Acad. Publ., 2004.
\bibitem{Bourn1991} D.~Bourn, \emph{Normalization equivalence, kernel equivalence and affine categories}, Category {T}heory, {P}roceedings {C}omo 1990 (A.~Carboni, M.~C. Pedicchio, and G.~Rosolini, eds.), Lecture Notes in Math., vol. 1488, Springer, 1991, pp.~43--62.
\bibitem{MCFPO} D.~Bourn, \emph{Mal'cev categories and fibration of pointed objects}, Appl. Categ. Structures \textbf{4} (1996), 307--327.
\bibitem{BGJ} D.~Bourn, M.~Gran and P.-A.~Jacqmin, \emph{On the naturalness of Mal'tsev categories}, in Joachim Lambek: the interplay of Mathematics, Logic and Linguistics, Outstanding Contributions to Logic, vol. 20, Eds. C. Casadio and P. Scott, Springer, 2021, pp.~59--104.

\bibitem{CKP}
A.~Carboni, G.~M. Kelly, and M.~C. Pedicchio, \emph{Some remarks on {M}altsev  and {G}oursat categories}, Appl. Categ. Structures \textbf{1} (1993), 385--421.
\bibitem{CLP} A.~Carboni, J.~Lambek, and M.~C. Pedicchio, \emph{Diagram chasing in {M}al'cev categories}, J.~Pure Appl. Algebra \textbf{69} (1991), 271--284.
\bibitem{CPP} A.~Carboni, M.~C. Pedicchio, and N.~Pirovano, \emph{Internal graphs and internal groupoids in {M}al'cev categories}, Am. Math. Soc. for the Canad. Math. Soc., Providence, 1992, pp.~97--109.
\bibitem{CCT} M.M.~Clementino, E.~Colebunders, W.~Tholen, \emph{Lax algebras as spaces}, in: Monoidal Topology, pp. 375--465. Encyclopedia Math. Appl. 153, Cambridge University Press, Cambridge, 2014.
\bibitem{CH03} M.~M.~Clementino and D.~Hofmann, \emph{Topological features of lax algebras}, Appl. Categ. Structures \textbf{11} (2003) 267--286.
\bibitem{CH17} M.M.~Clementino, D.~Hofmann, \emph{The rise and fall of
$V$-functors}, Fuzzy Sets and Systems \textbf{321} (2017) 29--49.
\bibitem{CM-FM} M.M. Clementino, N. Martins-Ferreira, A. Montoli, \emph{On the categorical behaviour of preordered groups}, {J.~Pure} Appl. Algebra \textbf{223} (2019) 4226--4245.
\bibitem{CT03} M.M.~Clementino, W.~Tholen, \emph{Metric, topology and multicategory --
 a common approach}, J. Pure Appl. Algebra \textbf{179} (2003) 13--47.
\bibitem{HM} J. Hagemann, A. Mitschke, \emph{On n-permutable congruences}, Algebra Univ. \textbf{3} (1973), 8--12.
\bibitem{HN23} D.~Hofmann, P.~Nora, \emph{Duality theory for enriched
Priestley spaces}, J. Pure Appl. Algebra \textbf{227} (2023) No. 107231.
\bibitem{ZJanelidze-Subtractive} Z.~Janelidze, \emph{Subtractive categories}, Appl. Categ. Structures \textbf{13} (2005), 343--350.
\bibitem{JRVdL} Z. Janelidze, D. Rodelo, T. Van der Linden,  \emph{Hagemann's theorem for regular categories}, Journal of Homotopy and Related Structures 9(1) (2014) 55--66.
\bibitem{Lambek} J. Lambek, \emph{Goursat's theorem and the Zassenhaus lemma}, Canad. J. Math. \textbf{10} (1957) 45--56.
\bibitem{Law} F.W.~Lawvere, \emph{Metric spaces, generalized logic, and
closed categories}, Rend. Sem. Mat. Fis. Milano \textbf{43} (1973)
135--166. Republished in: Reprints in Theory and Applications of
Categories, No. 1 (2002) 1-- 37.
\bibitem{Maltsev-Sbornik} A.I.~Mal'tsev, \emph{On the general theory of algebraic systems}, Mat. Sbornik N. S. \textbf{35}(77) (1954), 3--20 (1954).
\bibitem{MST} M.~Mahmoudi, C.~Schubert, W.~Tholen, \emph{Universality of
coproducts in categories of lax algebras}, Appl. Categ. Structures
\textbf{14} (2006) 243--249.
\bibitem{MF} N.~Martins-Ferreira, \emph{New wide class of weakly Mal'tsev categories}, Appl. Categ. Structures \textbf{23} (2015) 741--751.
\bibitem{M-FRVdL} N. Martins-Ferreira, D. Rodelo, T. Van der Linden, \emph{An observation on n-permutability}, Bull. Belg. Math. Soc. Simon Stevin \textbf{21} (2014), no. 2, 223--230.
\bibitem{MRVdL} A.~Montoli, D.~Rodelo, T.~Van der Linden, \emph{Two characterisations of groups amongst monoids}, J.~Pure Appl. Algebra \textbf{222} (2018) 747--777.
\bibitem{Smith} J.~D.~H. Smith, \emph{{M}al'cev varieties}, Lecture Notes in Math., vol. 554,  Springer, 1976.
\bibitem{Weighill} T.~Weighill, \emph{Mal'tsev objects, $R_1$-spaces and ultrametric spaces}, Theory Appl. Categ. \textbf{32} (2017), no.~42, 1485--1500.
\end{thebibliography}
\end{document}